\documentclass[reqno]{amsart}
\usepackage{mathptmx}
\usepackage{amssymb}
\usepackage{amsfonts}
\usepackage{amsmath}
\usepackage{graphicx}
\usepackage{shadow}
\usepackage{color}
\usepackage[all]{xy}
\usepackage{pxfonts}
\newtheorem{thm}{Theorem}[section]
\newtheorem{corollary}[thm]{Corollary}
\newtheorem{lemma}[thm]{Lemma}
\newtheorem{proposition}[thm]{Proposition}

\theoremstyle{definition}

\theoremstyle{remark}

\DeclareMathOperator{\htt}{ht}

\DeclareMathOperator{\Spec}{Spec}

\DeclareMathOperator{\m}{\frak{m}}
\DeclareMathOperator{\n}{\frak{n}}

\DeclareMathOperator{\edim}{embdim}
\DeclareMathOperator{\cdim}{codim}

\begin{document}

\title[Embedding dimension and codimension of tensor products of $k$-algebras]{Embedding dimension and codimension\\ of tensor products of algebras over a field}

\author[S. Bouchiba]{S. Bouchiba $^{(\star)}$}
\address{Department of Mathematics, University of Meknes, Meknes 50000, Morocco}
\email{bouchiba@fs-umi.ac.ma}

\author[S. Kabbaj]{S. Kabbaj $^{(\star)}$}
\address{Department of Mathematics and Statistics, King Fahd University of Petroleum \& Minerals (KFUPM), Dhahran 31261, KSA}
\email{kabbaj@kfupm.edu.sa}
\thanks{$^{(\star)}$ Supported by KFUPM under DSR Research Grant \# RG1212.}

\date{\today}



\dedicatory{To David Dobbs on the occasion of his 70th birthday}

\begin{abstract}
Let $k$ be a field. This paper investigates the embedding dimension and codimension of Noetherian local rings arising as localizations of tensor products of $k$-algebras. We use results and techniques from prime spectra and dimension theory to establish an analogue of the ``special chain theorem" for the embedding dimension of tensor products, with effective consequence on the transfer or defect of regularity as exhibited by the  (embedding) codimension given by $\cdim(R):=\edim(R)-\dim(R)$.
\end{abstract}
\maketitle

\section{Introduction}\label{i}

\noindent Throughout, all rings are commutative with identity elements, ring homomorphisms are unital, and $k$ stands for a field. The embedding dimension of a Noetherian local ring $(R,\m)$, denoted by $\edim(R)$, is the least number of generators of $\m$ or, equivalently, the dimension of $\m/\m^{2}$ as an $R/\m$-vector space. The ring $R$ is regular if its Krull dimension and embedding dimensions coincide. The (embedding) codimension of $R$ measures the defect of regularity of $R$ and is given by the formula $\cdim(R):=\edim(R)-\dim(R)$. The concept of regularity was initially introduced by Krull and became prominent when Zariski showed that a local regular ring corresponds to a smooth point on an algebraic variety. Later, Serre proved that a ring is regular if and only if it has finite global dimension. This allowed to see that regularity is stable under localization and then the definition got globalized as follows: a Noetherian ring is regular if its localizations with respect to all prime ideals are regular. The ring $R$ is a complete intersection if its $\m$-completion is the quotient ring of a local regular ring modulo an ideal generated by a regular sequence; $R$ is Gorenstein if its injective dimension is finite; and $R$ is Cohen-Macaulay if the grade and height of $\m$ coincide. All these algebro-geometric notions are globalized by carrying over to localizations.

These concepts transfer to tensor products of algebras over a field
under suitable assumptions. It has been proved that a Noetherian
tensor product of algebras (over a field) inherits the notions of
(locally) complete intersection ring, Gorenstein ring, and
Cohen-Macaulay ring \cite{BK2,HTY,TY,WITO}. In particular, a
Noetherian tensor product of any two extension fields is a complete
intersection ring. As to regularity and unlike the above notions, a
Noetherian tensor product of two extension fields of k is not
regular in general. In 1965, Grothendieck proved a positive result
in case one of the two extension fields is a finitely generated
separable extension \cite{Gr}. Recently, we have investigated the
possible transfer of regularity to tensor products of algebras over
a field $k$. If $A$ and $B$ are two $k$-algebras such that $A$ is
geometrically regular; i.e., $A\otimes_{k}F$ is regular for every
finite extension $F$ of $k$ (e.g., $A$ is a separable extension
field over $k$), we proved that  $A\otimes_kB$ is regular if and
only if $B$ is regular and $A\otimes_kB$ is Noetherian \cite[Lemma
2.1]{BK1}. As a consequence, we established necessary and sufficient
conditions for a Noetherian tensor product of two extension fields
of $k$ to inherit regularity under (pure in)separability conditions
\cite[Theorem 2.4]{BK1}. Also, Majadas'
relatively recent paper tackled questions of
regularity and complete intersection of tensor products of
commutative algebras  via the homology theory of Andr\'e and Quillen
\cite{Maj}. Finally, it is worthwhile recalling that tensor products of
rings subject to the above concepts were recently used to broaden or
delimit the context of validity of some homological conjectures; see
for instance \cite{HJ,J}. Suitable background on regular, complete intersection, Gorenstein,
and Cohen-Macaulay rings is \cite{BH,Gr,K,M}. For a geometric
treatment of these properties, we refer the reader to the excellent
book of Eisenbud \cite{E}.

Throughout, given a ring $R$,  $I$ an ideal of $R$ and $p$ a prime
ideal of $R$, when no confusion is likely, we will denote by $I_p$
the ideal $IR_p$ of the local ring $R_p$ and by $\kappa_R(p)$ the
residue field of $R_p$. One of the cornerstones of dimension theory of polynomial rings in several variables is \emph{the special chain theorem}, which essentially asserts that the height of any prime ideal of the polynomial  ring can always be realized via a special chain of prime ideals passing by the extension of its contraction over the basic ring; namely, if $R$ is a Noetherian ring and  $P$ is a prime ideal of $R[X_1,...,X_n]$ with $p:=P\cap R$, then
$$\dim(R[X_1,...,X_n]_P)=\dim(R_p)+\dim\left(\kappa_R(p)[X_1,...,X_n]_{\frac{P_p}{pR_{p}[X_1,...,X_n]}}\right)$$
An analogue of this result for Noetherian tensor products,
established in \cite{BK2}, states that, for any prime ideal $P$ of
$A\otimes_kB$ with $p:=P\cap A$ and $q:=P\cap B$, we have $$
\dim(A\otimes_kB)_P =\dim(A_p)\ +\
\dim\left(\Big(\kappa_A(p)\otimes_kB\Big)_{\frac{P_p}{pA_p\otimes_kB}}\right)
$$
which also comes in the following extended form
$$ \dim(A\otimes_kB)_{P} =
\dim(A_{p})+\dim(B_{q})+\dim\left(\big(\kappa_A(p)\otimes_k\kappa_B(q)\big)_{\frac{P(A_p\otimes_kB_q)}{pA_p\otimes_kB_{q}+A_p\otimes_kqB_{q}}}\right).
$$
This paper investigates the embedding dimension of Noetherian local
rings arising as localizations of tensor products of $k$-algebras.
We use results and techniques from prime spectra and dimension
theory to establish satisfactory analogues of the ``special chain
theorem" for the embedding dimension in various contexts of tensor
products, with effective consequences on the transfer or defect of
regularity as exhibited by the  (embedding) codimension. The paper
traverses four sections along with an introduction.

In Section 2, we introduce and study a new invariant which allows to correlate the embedding dimension of a Noetherian local ring $B$ with the fibre ring $B/\m B$ of a local homomorphism $f:A\longrightarrow B$ of Noetherian local rings. This enables us to provide an analogue of the special chain theorem for the embedding dimension as well as to generalize the known result that ``\emph{if $f$ is flat and $A$ and $B/\m B$ are regular rings, then $B$ is regular}."

Section~\ref{p} is devoted to the special case of polynomial rings which will be used in the investigation of tensor products.
The main result (Theorem~\ref{p:1}) states that, for a Noetherian ring $R$ and $X_1,...,X_n$ indeterminates over $R$, for any prime ideal $P$ of
 $R[X_1,...,X_n]$ with $p:=P\cap R$, we have:\\
$$
\begin{array}{lll}
\edim(R[X_1,...,X_n]_P)  &=&\edim(R_p)+\htt\left(\dfrac{P}{p[X_1,...,X_n]}\right)\\
                        &=&\edim(R_p)+\edim\left(\kappa_R(p)[X_1,...,X_n]_{\frac{P_p}{pR_{p}[X_1,...,X_n]}}\right)
\end{array}
$$
Then, Corollary~\ref{p:2} asserts that
$$\cdim(R[X_1,...,X_n]_P)=\cdim(R_p)$$ and recovers a well-known
result on the transfer of regularity to polynomial rings; i.e.,
\emph{$R[X_1,...,X_n]$ is regular if and only if so is $R$} (this
result was initially proved via Serre's result on finite global
dimension and Hilbert Theorem on syzygies). Then Corollary~\ref{p:3}
characterizes regularity in general settings of localizations of
polynomial rings and, in the particular cases of Nagata rings and
Serre conjecture rings, it states that \emph{$R(X_1,...,X_n)$ is
regular if and only if $R\langle X_1,...,X_n\rangle$ is regular if
and only if $R$ is regular}.

Let $A$ and $B$ be two $k$-algebras such that $A\otimes_kB$ is Noetherian and let $P$ be a prime ideal of $A\otimes_kB$ with $p:=P\cap A$ and $q:=P\cap B$. Due to known behavior of tensor products of $k$-algebras subject to
regularity (cf. \cite{BK1,Gr,HTY,TY,WITO}), Section~\ref{s}
investigates the case when $A$ (or $B$) is a separable (not
necessarily algebraic)  extension field of $k$. The main result
(Theorem~\ref{s:1}) asserts that, if $K$ is a separable extension
field of $k$, then
$$\edim(K\otimes_kA)_P   =\edim(A_p)\ +\  \edim\left(\Big(K\otimes_k\kappa_A(p)\Big)_{\frac{P_p}{K\otimes_kpA_p}}\right).$$
In particular, if $K$ is separable algebraic over $k$, then
$$\edim(K\otimes_kA)_{P}  = \edim(A_{p}).$$ Then, Corollary~\ref{s:2}
asserts that $$\cdim(K\otimes_kA)_{P}=\cdim(A_p)$$ and hence
$K\otimes_kA$ is regular if and only if so is $A$. This recovers
Grothendieck's result on the transfer of regularity to tensor
products issued from finite extension fields \cite[Lemma
6.7.4.1]{Gr}.

Section~\ref{r} examines the more general case of tensor products of
$k$-algebras with separable residue fields. The main theorem
(Theorem~\ref{r:1}) states that if $\kappa_B(q)$ is a separable
extension field of $k$, then
$$
\begin{array}{rl}
\edim(A\otimes_kB)_{P}\ =   &\edim(A_{p})\ +\ \edim(B_{q})\\
                            &\hspace{2.1cm}+\ \edim\left(\Big(\kappa_A(p)\otimes_k\kappa_B(q)\Big)_{\frac{P(A_p\otimes_kB_q)}{pA_p\otimes_kB_{q}+A_p\otimes_kqB_{q}}}\right)
\end{array}
$$
Then, Corollary~\ref{r:2} contends that
$$\cdim(A\otimes_kB)_{P}=\cdim(A_{p})+\cdim(B_{q})$$ recovering
known results on the transfer of regularity to tensor products over
perfect fields \cite[Theorem 6(c)]{TY} and, more generally, to
tensor products issued from residually separable extension fields
\cite[Theorem 2.11]{BK1}.

The four aforementioned main results are connected as follows:
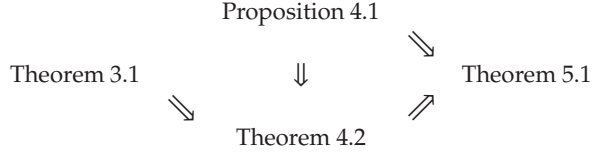
\[\begin{array}{ccccc}
                                &                   &\text{\small Proposition \ref{f:1}}    &           &\\
                                &                   &                                   &\Searrow   &\\
\text{\small Theorem \ref{p:1}} &                   &\Downarrow                         &           &\text{\small Theorem \ref{r:1}}\\
                                &\Searrow           &                                   &\Nearrow   &\\
                                &                   &\text{\small Theorem \ref{s:1}}    &           &
\end{array}\]

Of relevance to this study is Bouchiba, Conde-Lago, and Majadas' recent preprint \cite{BCM} where the authors prove some of our results via the homology theory of Andr\'e and Quillen. In the current paper, we offer direct and self-contained proofs using techniques and basic results from commutative ring theory. Early and recent developments on prime spectra and dimension theory are to be found in \cite{BDK,BGK1,BGK2,BK2,S,S2,S3,V,W} for the special case of tensor products of $k$-algebras, and in \cite{ABDFK,BDF,FK,J,K,M,Na}  for the general case. Any unreferenced material is standard, as in \cite{K,M}.

\section{Embedding dimension of Noetherian local rings}

In this section, we discuss the relationship between the embedding dimensions of Noetherian local rings connected by a local ring homomorphism. To this purpose, we introduce a new invariant $\mu$ which allows to relate the embedding dimension of a local ring to that of its fibre ring.

Throughout, let $(A,\m,K)$ and $(B,\n,L)$ be local Noetherian rings, $f:A\longrightarrow B$ a local homomorphism (i.e., $\m B:=f(\m)B\subseteq\n$), and $I$ a proper ideal  of $A$. Let
$$\mu_A(I):=\dim_K\left(\dfrac{I+\m^2}{\m^2}\right).$$
Note that $\mu_A(I)$ equals the maximal number of elements of $I$ which are part of a minimal basis of $\m$; so that $0\leq\mu_A(I)\leq\edim(A)$ and $\mu_A(\m)=\edim(A)$. Next, let $\mu_B^f(I)$ denote the maximal number of elements of $IB:=f(I)B$ which are part of a minimal basis of $\n$; that is,
$$\mu_B^f(I):=\mu_B(IB)=\dim_L\left(\dfrac {IB+\n^2}{\n^2}\right).$$
It is easily seen that if $x_1,\dots,x_r$ are elements of $\m$ such that $f(x_1),\dots,f(x_r)$ are part of a minimal basis of $\n$, then $x_1,\dots,x_r$ are part of a minimal basis of $\m$ as well. That is, $0\leq \mu_B^f(I)\leq \mu_A(I)$. Moreover, if $J$ is a proper ideal of $B$ and $\pi:B\twoheadrightarrow B/J$ is the canonical surjection, then the natural linear map  of $L$-vector spaces
$\dfrac{IB+\n^2}{\n^2}\twoheadrightarrow \dfrac{IB+\n^2+J}{\n^2+J}$
yields $\mu_{B/J}^{\pi\circ f}(I)\leq \mu_B^f(I)$.

\begin{proposition}\label{n:1}
Under the above notation, we have:
$$\edim(B)=\mu_B^f(I)+\edim(B/IB).$$
In particular, $$\edim(A)=\mu_A(I)+\edim(A/I).$$
\end{proposition}

\begin{proof} The first statement follows easily from the following exact sequence of
$L$-vector spaces
$$0\longrightarrow \dfrac{IB+n^2}{n^2}\longrightarrow \dfrac n{n^2}\longrightarrow \dfrac{n}{IB+n^2}=\dfrac{n/IB}{(n/IB)^2}\longrightarrow 0.$$
The second statement holds since $\mu_A(I)=\mu_A^{\text{id}_A}(I)$.
\end{proof}

Recall that, under the above notation, the following inequality always holds: $\dim(B)\leq\dim(A)+\dim(B/\m B)$. The first corollary provides an analogue
for the embedding dimension.

\begin{corollary}\label{n:1.1}
Under the above notation, we have:
$$\edim(B)\leq\edim(A)-\edim(A/I)+\edim(B/IB).$$
In particular,
$$\edim(B)\leq\edim(A)+\edim(B/\m B).$$
\end{corollary}

It is well known that if $f$ is flat and both $A$ and $B/m\ B$ are regular, then $B$ is regular. The second corollary generalizes this result to homomorphisms
subject to going-down. Recall that a ring homomorphism $h:R\longrightarrow S$ satisfies going-down (henceforth abbreviated GD) if for any pair $p\subseteq q$ in $\Spec(R)$ such that there exists $Q\in\Spec(S)$ lying over $q$, then there exists $P\in\Spec(S)$ lying over $p$ with $P\subseteq Q$. Any flat ring homomorphism satisfies GD.

\begin{corollary}
 Under the above notation, assume that $f$ satisfies GD. Then:
\begin{enumerate}
\item $\cdim(B)=\left(\mu_B^f(\m)-\dim(A)\right)+\cdim(B/\m B)$.

\item $\cdim(B)+\left(\edim(A)-\mu_B^f(\m)\right)=\cdim(A)+\cdim(B/\m B)$.

\item $B$ is regular and $\mu_B^f(\m)=\edim(A)$ $\Longleftrightarrow$ $A$ and
$B/\m B$ are regular.
\end{enumerate}
\end{corollary}

\begin{proof}
The proof is straightforward via a combination of Proposition~\ref{n:1} and \cite[Theorem 15.1]{M}.
\end{proof}

\begin{corollary}
 Under the above notation, assume that $f$ satisfies GD. Then:
\begin{enumerate}
\item $\cdim(B)\leq \cdim(A)+\cdim(B/\m B)$.
\item If $B/\m B$ is regular, then $\cdim(B)\leq\cdim(A)$.
\end{enumerate}
\end{corollary}

\begin{proof}
The proof is direct via a combination of Corollary~\ref{n:1.1} and the known fact that $\dim(B)=\dim(A)+\dim(B/\m B)$.
\end{proof}

\section{Embedding dimension and codimension of polynomial rings}\label{p}

\noindent This section is devoted to the special case of polynomial
rings which will be used, later, for the investigation of tensor
products. The main result of this section (Theorem~\ref{p:1})
settles a formula for the embedding dimension for the localizations
of polynomial rings over  Noetherian rings. It recovers (via
Corollary~\ref{p:2}) a well-known result on the transfer of
regularity to polynomial rings; that is, $R[X_1,....,X_n]$ is
regular if and only if so is $R$. Moreover, Theorem~\ref{p:1} leads
to investigate the regularity of two famous localizations of
polynomial rings in several variables; namely, the Nagata ring
$R(X_1,X_2,...,X_n)$ and Serre conjecture ring $R\langle
X_1,X_2,...,X_n\rangle$. We show that the regularity of these two
constructions is entirely characterized by the regularity of $R$
(Corollary~\ref{p:3}).

Recall that one of the cornerstones of dimension theory of
polynomial rings in several variables is \emph{the special chain
theorem}, which essentially asserts that the height of any prime
ideal $P$ of $R[X_1,...,X_n]$ can always be realized via a special
chain of prime ideals passing by the extension $(P\cap
R)[X_1,...,X_n]$. This result was first proved by Jaffard in
\cite{J} and, later, Brewer, Heinzer, Montgomery and Rutter
reformulated it in the following simple way (\cite[Theorem
1]{BHMR}): Let $P$ be a prime ideal of $R[X_1,...,X_n]$ with
$p:=P\cap R$. Then
$\htt(P)=\htt(p[X_1,...,X_n])+\htt\left(\dfrac{P}{p[X_1,...,X_n]}\right).$
In a Noetherian setting, this formula becomes:
\begin{equation}\label{sct}
\begin{array}{lll}
\dim(R[X_1,...,X_n]_P)  &=&\dim(R_p)+\htt\left(\dfrac{P}{p[X_1,...,X_n]}\right)\\
                        &=&\dim(R_p)+\dim\left(\kappa_R(p)[X_1,...,X_n]_{\frac{P_p}{pR_{p}[X_1,...,X_n]}}\right)
\end{array}
\end{equation}
where the second equality holds on account of the basic fact
$\frac{P}{p[X_1,...,X_n]}\cap\frac{R}{p}=0$. The main result of this
section (Theorem~\ref{p:1}) features a ``special chain theorem" for
the embedding dimension with effective consequence on the
codimension.\\

\begin{thm}\label{p:1}
Let $R$ be a Noetherian ring and $X_1,...,X_n$ be indeterminates
over $R$. Let $P$ be a prime ideal of $R[X_1,...,X_n]$ with
$p:=P\cap R$. Then:
$$
\begin{array}{lll}
\edim(R[X_1,...,X_n]_P)  &=&\edim(R_p)+\htt\left(\dfrac{P}{p[X_1,...,X_n]}\right)\\
                        &=&\edim(R_p)+\edim\left(\kappa_R(p)[X_1,...,X_n]_{\frac{P_p}{pR_{p}[X_1,...,X_n]}}\right)
\end{array}
$$
\end{thm}

\begin{proof} We use induction on $n$. Assume $n=1$ and let $P$ be a prime ideal of $R[X]$ with $p:=P\cap R$ and $r:=\edim(R_p)$. Then $p_{p}=(a_{1}, ..., a_{r})R_p$ for some $a_{1}, ..., a_{r}\in p$. We envisage two cases; namely, either $P$ is an extension of $p$ or an upper to $p$. For both cases, we will use induction on $r$.

{\bf Case 1:} $P$ is an extension of $p$ (i.e., $P=pR[X]$). We prove
that  $\edim(R[X]_P)$ $=r$. Indeed, we have $P_{P} =
pR_p[X]_{pR_p[X]}=(a_{1}, ..., a_{r})R_p[X]_{pR_p[X]}=(a_{1}, ...,
a_{r})R[X]_P.$ So, obviously, if $p_p=(0)$, then $P_{P}=0$. Next, we
may assume $r\geq 1$. One can easily check that the canonical ring
homomorphism $\varphi:R_p\longrightarrow R[X]_P$ is injective with
$\varphi(p_p)\subseteq P_P$. This forces $\edim(R[X]_P)\geq 1$.
Hence, there exists $j\in \{1,...,n\}$, say $j=1$, such that
$a:=a_{1}\in p$ with $\frac{a}{1}\in P_P\setminus P^2_P$ and, a
fortiori, $\frac{a}{1}\in p_p\setminus p^2_p$. By \cite[Theorem
159]{K}, we get
\begin{equation}\label{Kap159}
\left\{
\begin{array}{lll}
\edim(R[X]_P)=1+\edim\left(\dfrac R{(a)}[X]_{\frac P{aR[X]}}\right)\\
\edim(R_p)=1+ \edim\left((\dfrac R{(a)}\Big )_{\frac p{(a)}}\right)
\end{array}\right.
\end{equation}
Therefore $\edim\left(\Big(\dfrac R{(a)}\Big)_{\frac
p{(a)}}\right)=r-1$ and then, by induction on $r$, we obtain
\begin{equation}\label{ind1}\edim\left(\frac R{(a)}[X]_{\frac P{aR[X]}}\right)=\edim\left(\Big(\frac R{(a)}\Big)_{\frac p{(a)}}\right).\end{equation}
A combination of (\ref{Kap159}) and (\ref{ind1}) leads to
$\edim(R[X]_P) = r$, as desired.

{\bf Case 2:} $P$ is an upper to $p$ (i.e., $P\not=pR[X]$). We prove
that  $\edim(R[X]_P) = r+1$. Note that $PR_p[X]$ is also an upper to
$p_p$ and then there exists a (monic) polynomial $f\in R[X]$ such
that $\overline{\frac{f}{1}}$ is irreducible in $\kappa_R(p)[X]$ and
$PR_p[X]=pR_p[X]+fR_p[X]$. Notice that $pR[X]+fR[X]\subseteq P$ and
we have
$$\begin{array}{rcl}
P_P &=  &PR_p[X]_{PR_p[X]}=(pR_p[X]+fR_p[X])_{PR_p[X]}\\
    &=  &(p[X]+fR[X])R_p[X]_{PR_p[X]}=(p[X]+fR[X])_{P}\\
    &=  &p[X]_P+fR[X]_P=(a_{1}, ..., a_{r}, f)R[X]_P.
\end{array}
$$
Assume $r=0$. Then $P$ is an upper to zero with  $P_{P}=fR[X]_{P}$.
So that $\edim(R[X]_P)\leq 1$. Further, by the principal ideal theorem
\cite[Theorem 152]{K}, we have
$$\edim(R[X]_P) \geq \dim(R[X]_P)=\htt(P)=1.$$
It follows that  $\edim(R[X]_P)= 1$, as desired.

Next, assume $r\geq 1$. We claim that $pR[X]_{P}\nsubseteqq
P^{2}_{P}$. Deny and suppose that $pR[X]_{P}\subseteq P^{2}_{P}$.
This assumption combined with the fact $P_P=p[X]_P+fR[X]_P$ yields
$\dfrac{P_P}{P^2_P}=\overline{f}R[X]_P$ as $R[X]_P$-modules and
hence $P_P=fR[X]_P$ by \cite[Theorem 158]{K}. Next, let $a\in p$.
Then, as $\frac a1\in P_P=fR[X]_P$, there exist $g\in R[X]$ and $s,
t\in R[X]\setminus P$ such that $t(sa-fg)=0$. So that $tfg\in p[X]$,
whence $tg\in p[X]\subset P$ as $f\notin p[X]$. It follows that $tsa
= tfg\in P^{2}$ and thus $\frac a1\in P_P^2=f^2R[X]_P$. We iterate
the same process to get $\frac a1\in P_P^n=f^nR[X]_P$ for each
integer $n\geq 1$. Since $R[X]_P$ is a Noetherian local ring,
$\bigcap P^n_P=(0)$ and thus ${\frac a1}=0$  in $R[X]_P$. By the
canonical injective homomorphism $R_p\hookrightarrow R[X]_P$,
$\frac{a}{1}=0$ in $R_p$. Thus $p_p=(0)$, the desired contradiction.

Consequently, $pR[X]_{P}=(a_{1}, ..., a_{r})R[X]_P\nsubseteqq
P^{2}_{P}.$ So, there exists $j\in \{1,...,n\}$, say $j=1$, such
that $a:=a_{1}\in P_P\setminus P^2_P$ and, a fortiori, $a\in
p_p\setminus p^2_p$. Similar arguments as in Case 1 lead to the same
two formulas displayed in (\ref{Kap159}). Therefore
$\edim\left(\Big(\dfrac R{(a)}\Big)_{\frac p{(a)}}\right)=r-1$ and
then, by induction on $r$, we obtain
\begin{equation}\label{ind2}\edim\left(\frac{R}{(a)}[X]_{\frac{P}{aR[X]}}\right)=1+\edim\left(\Big(\frac{R}{(a)}\Big)_{\frac{p}{(a)}}\right).\end{equation}
A combination of (\ref{Kap159}) and (\ref{ind2}) leads to
$\edim(R[X]_P) = r + 1$, as desired.

Now, assume that $n\geq 2$ and set $R[k]:=R[X_1, ...,X_k]$ and
$p[k]=p[X_1, ...,X_k]$ for $k:=1,...,n$. Let $P^{\prime}:=P\cap
R[n-1]$. We prove that $\edim(R[n]_P) = r+\htt\left(\dfrac
P{p[n]}\right).$ Indeed, by virtue of the case $n=1$, we have
\begin{equation}\label{n=1}\edim(R[n]_P)=\edim(R[n-1]_{P^{\prime}})+\htt\left(\frac P{P^{\prime}[X_n]}\right).\end{equation}
Moreover, by induction hypothesis, we get
\begin{equation}\label{ind3}\edim(R[n-1]_{P^{\prime}})=r+\htt\left(\frac {P^{\prime}}{p[n-1]}\right).\end{equation}
Note that the prime ideals $\dfrac{P^{\prime}[X_{n}]}{p[n]}$ and
${\dfrac {P}{p[n]}}$ both survive in $\kappa_R(p)[n]$, respectively.
Hence, as  $\kappa_R(p)[n]$ is catenarian and $(R/p)[n-1]$ is
Noetherian, we obtain
\begin{equation}\label{cat}
\begin{array}{rcl}
\htt\left(\dfrac P{p[n]}\right)=
\htt\left(\dfrac{P^{\prime}[X_{n}]}{p[n]}\right)+\htt\left(\dfrac{P}{P^{\prime}[X_n]}\right)
=\htt\left(\dfrac{P^{\prime}}{p[n-1]}\right) +
\htt\left(\dfrac{P}{P^{\prime}[X_n]}\right).
\end{array}
\end{equation}
Further, the fact that $\kappa_R(p)[X_1,...,X_n]$ is regular  yield
\begin{equation}\label{reg1}
\htt\left(\dfrac{P}{p[X_1,...,X_n]}\right)=
\edim\left(\kappa_R(p)[X_1,...,X_n]_{\frac{P_p}{pR_{p}[X_1,...,X_n]}}\right).
\end{equation}
So (\ref{n=1}), (\ref{ind3}), (\ref{cat}), and (\ref{reg1}) lead to
the conclusion, completing the proof of the theorem.
\end{proof}

As a first application of Theorem~\ref{p:1}, we get the next
corollary on the (embedding) codimension. In particular, it recovers
a well-known result on the transfer of regularity to polynomial
rings (initially proved via Serre's result on finite global
dimension and Hilbert Theorem on syzygies \cite[Theorem 8.37]{Rot}.
See also \cite[Theorem 171]{K}).

\begin{corollary}\label{p:2}
Let $R$ be a Noetherian ring and $X_1,...,X_n$ be indeterminates
over $R$. Let $P$ be a prime ideal of $R[X_1,...,X_n]$ with
$p:=P\cap R$. Then:
$$\cdim(R[X_1,...,X_n]_P)=\cdim(R_p).$$ In particular, $R[X_1,...,X_n]$ is regular if and only if $R$ is regular.
\end{corollary}

Theorem~\ref{p:1} allows us to characterize the regularity for two
famous localizations of polynomial rings; namely, Nagata rings and
Serre conjecture rings. Let $R$ be a ring and $X,X_1,...,X_n$
indeterminates over $R$. Recall that
$R(X_1,...,X_n)=S^{-1}R[X_1,...,X_n]$ is the Nagata ring, where $S$
is the multiplicative set of $R[X_1,...,X_n]$ consisting of the
polynomials whose coefficients generate $R$. Let $R\langle
X\rangle:=U^{-1}R[X]$, where $U$ is the multiplicative set of monic
polynomials in $R[X]$, and $R\langle X_1,\cdots,X_n\rangle:=R\langle
X_1,...,X_{n-1}\rangle\langle X_n\rangle$. Then $R\langle
X_1,...,X_n\rangle$ is called the Serre conjecture ring and is a
localization of $R[X_1,...,X_n]$.

\begin{corollary}\label{p:3}
Let $R$ be a Noetherian ring and $X_1,...,X_n$ indeterminates over
$R$. Let $S$ be a multiplicative subset of $R[X_1, ...,X_n]$.  Then:
\begin{enumerate}
\item $S^{-1}R[X_1,...,X_n]$ is regular if and only if $R_p$ is regular for each prime ideal $p$ of $R$ such that $p[X_1,...,X_n]\cap S=\emptyset$.
\item In particular, $R(X_1,...,X_n)$ is regular if and only if $R\langle X_1,...,X_n\rangle$ is regular if and only if $R[X_1,...,X_n]$ is regular if and only if $R$ is regular.
\end{enumerate}
\end{corollary}

\begin{proof}
(a) Let $Q=S^{-1}P$ be a prime ideal of $S^{-1}R[X_1,...,X_n]$,
where $P$ is the inverse image of $Q$ by the canonical homomorphism
$R[X_1,...,X_n]\rightarrow S^{-1}R[X_1,...,X_n]$ and let $p:=P\cap
R$. Notice that $S^{-1}R[X_1,...,X_n]_{Q}\cong R[X_1,...,X_n]_P$
and\\
${\dfrac{Q}{S^{-1}p[X_1,...,X_n]}}\cong{\overline{S}}^{-1}{\dfrac
P{p[X_1,...,X_n]}}$ where $\overline{S}$ denotes the image of $S$
via the natural homomorphism $R[X_1,...,X_n]\rightarrow {\frac
Rp}[X_1,...,X_n].$ Therefore, by (\ref{sct}), we obtain
\begin{equation}\label{dim}
\dim(S^{-1}R[X_1,...,X_n]_{Q})
=\dim(R[X_1,...,X_n]_P)=\dim(R_p)+{\htt\Big(\dfrac{Q}{S^{-1}p[X_1,...,X_n]}\Big)}
\end{equation}
and, by Theorem~\ref{p:1}, we have
\begin{equation}\label{edim}
\begin{array}{lll}
\edim(S^{-1}R[X_1,...,X_n]_{Q}) &=&\edim(R[X_1,...,X_n]_P)\\
                                &=&\edim(R_p)+{\htt\Big(\dfrac{Q}{S^{-1}p[X_1,...,X_n]}\Big)}.
\end{array}
\end{equation}

Now, observe that the set $\{Q\cap R\mid Q$ is a prime ideal of
$S^{-1}R[X_1,...,X_n]\}$ is equal to the set $\{p\mid p$ is a prime
ideal of $R$ such that $p[X_1,...,X_n]\cap S=\emptyset\}$.
Therefore, (\ref{dim}) and (\ref{edim}) lead to the conclusion.

(b) Combine (a) with the fact that the extension of any prime ideal
of $R$ to $R[X_1,...,X_n]$ does not meet the multiplicative sets
related to the rings $R(X_1,...,X_n)$ and  $R\langle
X_1,...,X_n\rangle$.
\end{proof}

\section{Embedding dimension and codimension of tensor products issued from separable extension fields}\label{s}

\noindent This section establishes an analogue of the ``special
chain theorem" for the embedding dimension of Noetherian tensor
products issued from separable extension fields, with effective
consequences on the transfer or defect of regularity. Namely, due to
known behavior of a tensor product $A\otimes_kB$ of two $k$-algebras
subject to regularity (cf. \cite{BK1,Gr,HTY,M,TY,WITO}), we will
investigate the case where $A$ or $B$ is a separable (not
necessarily algebraic) extension field of $k$.

Throughout, let $A$ and $B$ be two $k$-algebras such that
$A\otimes_kB$ is Noetherian and let $P$ be a prime ideal of
$A\otimes_kB$ with $p:=P\cap A$ and $q:=P\cap B$. Recall that $A$
and $B$ are Noetherian too; and the converse is not true, in
general, even if $A=B$ is an extension field of $k$ (cf.
\cite[Corollary 3.6]{Fer} or \cite[Theorem 11]{V}). We assume
familiarity with the natural isomorphisms for tensor products and
their localizations as in  \cite{B4-7,BAC8-9,Rot}. In particular, we
identify $A$ and $B$ with their respective images in $A\otimes_kB$
and we have $\dfrac {A\otimes_kB}{p\otimes_kB+A\otimes_kq}\cong
\dfrac Ap \otimes_k\dfrac Bq$ and $A_p\otimes_kB_q\cong
S^{-1}(A\otimes_kB)$ where $S:=\{s\otimes t\mid s\in A\setminus p,
t\in B\setminus q\}$. Throughout this and next sections, we adopt the following simplified notation for the invariant $\mu$:
$$\mu_P(pA_p):=\mu_{(A\otimes_kB)_P}^{i}(pA_p)\ \text{ and }\
\mu_P(qA_q):=\mu_{(A\otimes_kB)_P}^{j}(qB_q)$$
where $i:A_p\longrightarrow (A\otimes_kB)_P$ and $j:B_q\longrightarrow (A\otimes_kB)_P$ are the canonical (local flat) ring homomorphisms.

Recall that $A\otimes_kB$ is Cohen-Macaulay (resp., Gorenstein,
locally complete intersection) if and only if so are $A$ and the
fibre rings $\kappa_A(p)\otimes_kB$ (for each prime ideal $p$ of
$A$) \cite{BK2,TY}. Also if $A$ and the fibre rings
$\kappa_A(p)\otimes_kB$ are regular then so is  $A\otimes_kB$
\cite[Theorem 23.7(ii)]{M}. However, the converse does not hold in
general; precisely, if $A\otimes_kB$ is regular then so is $A$
\cite[Theorem 23.7(i)]{M} but the fibre rings are not necessarily
regular (see \cite[Example 2.12(iii)]{BK1}).

From \cite[Proposition 2.3]{BK2} and its proof, recall an analogue
of the special chain theorem (recorded in (\ref{sct})) for the tensor
products which correlates the dimension of $(A\otimes_kB)_{P}$ to
the dimension of its fibre rings; namely,
\begin{equation}\label{scte}
\begin{array}{lll}
\dim(A\otimes_kB)_P    &=&\dim(A_p)\ +\ \htt{\left(\dfrac P{p\otimes_kB}\right)}\\
                        &=&\dim(A_p)\ +\  \dim\left(\Big(\kappa_A(p)\otimes_kB\Big)_{\frac{P_p}{pA_p\otimes_kB}}\right)
\end{array}
\end{equation}

Our first result reformulates Proposition~\ref{n:1} and thus gives
an analogue of the special chain theorem for the embedding dimension
in the context of tensor products of algebras over a field.

\begin{proposition}\label{f:1}
Let $A$ and $B$ be two $k$-algebras such that  $A\otimes_kB$ is
Noetherian and let $P$ be a prime ideal of $A\otimes_kB$ with
$p:=P\cap A$ and $q:=P\cap B$. Then:
\begin{enumerate}
\item $\edim(A\otimes_kB)_{P}=\mu_P(pA_{p}) +
\edim\left(\Big(\kappa_A(p)\otimes_kB\Big)_{\frac{P_p}{pA_p\otimes_kB}}\right)$.

\item $\cdim(A\otimes_kB)_{P} +
\left(\edim(A_{p})-\mu_P(pA_{p})\right)=$

\hfill$\cdim(A_{p}) +\cdim\left(\big(\kappa_A(p)\otimes_kB\big)_{\frac{P_p}{pA_p\otimes_kB}}\right)$.

\item $(A\otimes_kB)_{P}$ is regular and
$\mu_P(pA_{p})=\edim(A_{p})$ if and only if both $A_{p}$ and
$\Big(\kappa_A(p)\otimes_kB\Big)_{\frac{P_p}{pA_p\otimes_kB}}$ are
regular.
\end{enumerate}
\end{proposition}

Recall that an extended form of the special chain theorem \cite{BK2}
states that
$$\dim(A\otimes_kB)_{P}=\dim(A_{p})+\dim(B_{q})+\dim\left(\Big(\kappa_A(p)\otimes_k\kappa_B(q)\Big)_{\frac{P(A_p\otimes_kB_q)}{pA_p\otimes_kB_{q}+
A_p\otimes_kqB_{q}}}\right).$$
In this vein, notice that, via Proposition~\ref{f:1}(a), we always
have the following inequalities:
$$
\begin{array}{lll}
\edim(A\otimes_kB)_{P}  &\leq   &\edim(A_{p})\ +\ \edim\left(\Big(\kappa_A(p)\otimes_kB\Big)_{\frac{P_p}{pA_p\otimes_kB}}\right)\\
                        &\leq   &\edim(A_{p})\ +\ \edim(B_{q})\\
                        &       & \hspace{1.4cm}+\ \edim\left(\Big(\kappa_A(p)\otimes_k\kappa_B(q)\Big)_{\frac{P(A_p\otimes_{k}B_q)}{pA_p\otimes_kB_q +
                        A_p\otimes_{k} qB_q}}\right).
\end{array}
$$

Let us state the main theorem of this section.

\begin{thm}\label{s:1}
Let $K$ be a separable extension field of $k$ and $A$ a $k$-algebra
such that $K\otimes_kA$ is Noetherian. Let $P$ be a prime ideal of
$K\otimes_kA$ with $p:=P\cap A$. Then:
$$
\begin{array}{lll}
\edim(K\otimes_kA)_P    &=&\edim(A_p)\ +\ \htt{\left(\dfrac P{K\otimes_kp}\right)}\\
                        &=&\edim(A_p)\ +\  \edim\left(\Big(K\otimes_k\kappa_A(p)\Big)_{\frac{P_p}{K\otimes_kpA_p}}\right)
\end{array}
$$
If, in addition, $K$ is algebraic over $k$, then
$\edim(K\otimes_kA)_{P} = \edim(A_{p}).$
\end{thm}

The proof of this theorem requires the following two preparatory
lemmas; the first of which determines a formula for the embedding
dimension of the tensor product of two $k$-algebras $A$ and $B$
localized at a special prime ideal $P$ with no restrictive
conditions on $A$ or $B$.

\begin{lemma}\label{s:1.1}
Let $A$ and $B$ be two $k$-algebras such that  $A\otimes_kB$ is
Noetherian and let $P$ be a prime ideal of $A\otimes_kB$ with
$p:=P\cap A$ and $q:=P\cap B$. Assume that
$P_{P}=(p\otimes_{k}B+A\otimes_{k}q)_{P}$. Then:
\begin{enumerate}
\item $\mu_P(pA_{p})=\edim(A_{p})$ and $\mu_P(qB_{q})=\edim(B_{q})$.
\item $\edim(A\otimes_kB)_{P}=\edim(A_{p})+\edim(B_{q}).$
\end{enumerate}
\end{lemma}

\begin{proof} We proceed through two steps.

\noindent{\bf Step 1.} Assume that $K:=B$ is an extension field of
$k$. Then $q=(0)$ and $P_P=p_p(A_p\otimes_kK)_{P_p}$. Let
$n:=\edim(A_p)$ and let $a_{1}, ..., a_{n}$ be elements of $p$ such
that $p_p=\left(\frac{a_{1}}{1}, ...,\frac{a_{n}}{1}\right)A_p.$ Our
argument uses induction on $n$. If $n=0$, then $A_p$ is a field and
$p_p=(0)$; hence $P_P=(0)$, whence $\edim(A\otimes_kK)_P=0$, as
desired. Next, suppose $n\geq 1$. We have
$P_P=\left(\frac{a_{1}}{1},
...,\frac{a_{n}}{1}\right)(A\otimes_kK)_P.$ If
$\edim(A\otimes_kK)_P=0$, $(A\otimes_kK)_P$ is regular and so is
$A_p$ by \cite[Theorem 23.7(i)]{M}. Hence, $n=\dim(A_{p})=0$ by
(\ref{scte}). Absurd. So, necessarily, $\edim(A\otimes_kK)_P\geq 1$.
Without loss of generality, we may assume that $\frac{a_{1}}{1}\in
P_P\setminus P^2_P$. Note that we already have  $\frac{a_{1}}{1}\in
p_p\setminus p^2_p$. Now, $\dfrac{P}{(a_1)\otimes_kK}$ is a prime
ideal of $\dfrac A{(a_1)}\otimes_kK$ with
$\dfrac{P}{(a_1)\otimes_kK}\cap \dfrac{A}{(a_1)}=\dfrac{p}{(a_1)}.$
By  \cite[Theorem 159]{K}, we obtain $\edim\left(\left(\dfrac
A{(a_1)}\right)_{\frac p{(a_1)}}\right)=n-1.$ By induction, we get
$$\edim\left(\left(\dfrac A{(a_1)}\otimes_kK\right)_{\frac{P}{(a_1)\otimes_kK}}\right)=\edim\left(\left(\dfrac A{(a_1)}\right)_{\frac p{(a_1)}}\right).$$
We conclude, via \cite[Theorem 159]{K}, to get
$$\edim(A\otimes_kK)_P=1+\edim\left(\left(\dfrac A{(a_1)}\otimes_kK\right)_{\frac{P}{(a_1)\otimes_kK}}\right)=n.$$
Moreover, observe that $\Big(\kappa_{A}(p)\otimes_kK\Big)_{\frac
{P_p}{p_p\otimes_kK}}$ is a field as $P_P=(p\otimes_kK)_P$. By
Proposition~\ref{f:1}, we have
\begin{equation}\label{s:1.1eq1}\mu_P(pA_p)=\edim(A\otimes_kK)_P=\edim(A_p).\end{equation}

\noindent{\bf Step 2.} Assume that $B$ is an arbitrary $k$-algebra.
Since $P_{P}=(p\otimes_{k}B+A\otimes_{k}q)_{P}$, then
$P(A_p\otimes_kB_q)=pA_p\otimes_kB_{q}+A_p\otimes_kqB_{q}$, hence
$\Big(\kappa_A(p)\otimes_k\kappa_B(q)\Big)_{\frac{P(A_p\otimes_kB_q)}{pA_p\otimes_kB_{q}+A_p\otimes_kqB_{q}}}$
is an extension field of $k$. So, apply Proposition~\ref{f:1} twice
to get
\begin{equation}\label{s:1.1eq2}\edim(A\otimes_kB)_P=\mu_P(qB_q)+\mu_{\frac{P_{q}}{A\otimes_{k}qB_{q}}}(pA_p).\end{equation}
Further, notice that
$$\begin{array}{lll} {\Big(\dfrac
{P_q}{A\otimes_kqB_q}}\Big )_{\frac {P_q}{A\otimes_kqB_q}}&=&
{\dfrac{(P_q)_{P_q}}{(A\otimes_kqB_q)_{P_q}}}={\dfrac
{P_P}{(A\otimes_kq)_P}}=
{\dfrac{(p\otimes_kB+A\otimes_kq)_P}{(A\otimes_kq)_P}}\\
 &=&\Big (
{\dfrac{p\otimes_kB+A\otimes_kq}{A\otimes_kq}}\Big )_{\frac
P{A\otimes_kq}}\cong\Big ( {p\otimes_k\dfrac Bq}\Big )_{\frac
P{A\otimes_kq}}={\Big (p\otimes_k\kappa_{B}(q)}\Big )_{\frac
{P_q}{A\otimes_kqB_q}}.\end{array}$$
Therefore, by (\ref{s:1.1eq1}), we get
$$\mu_{\frac{P_{q}}{A\otimes_{k}qB_{q}}}(pA_p)=\edim\left(\Big({A\otimes_k\kappa_{B}(q)}\Big)_{\frac{P_q}{A\otimes_kqB_q}}\right)=\edim(A_p).$$
Similar arguments yield
$$\mu_{\frac{P_{p}}{pA_{p}\otimes_{k}B}}(qB_q)=\edim\left(({\kappa_{A}(p)}\otimes_kB)_{\frac{P_p}{pA_p\otimes_kB}}\right)=\edim(B_q)$$
and, by the facts $0\leq\mu_P(pA_{p})\leq\edim(A_{p})$ and $\mu_{\frac{P_p}{pA_p\otimes_kB}}(qB_{q})\leq\mu_P(qB_{q})$, we obtain
$$\mu_P(pA_{p})=\edim(A_{p})\ \text{ and }\ \mu_P(qB_{q})=\edim(B_{q})$$
completing the proof of the lemma via (\ref{s:1.1eq2}).
\end{proof}

The second lemma will allow us to reduce our investigation to
tensor products issued from finite extension fields.

\begin{lemma}\label{s:1.2}
Let $K$ be an extension field of $k$ and $A$ a $k$-algebra such that
$K\otimes_kA$ is Noetherian. Let $P$ be a prime ideal of
$K\otimes_kA$. Then, there exists a finite extension field $E$ of
$k$ contained in $K$ such that
$$\edim(K\otimes_kA)_{P}=\edim(F\otimes_kA)_{Q}$$
for each intermediate field $F$ between $E$ and $K$ and $Q:=P\cap
(F\otimes_kA)$.
\end{lemma}

\begin{proof}
Let  $z_1,...,z_t\in K\otimes_kA$ such that
$P=\left(z_1,...,z_t\right)K\otimes_kA$; and for each $i=1,...,t$,
let $z_i:=\sum\limits_{j=1}^{n_i}\alpha_{ij}\otimes_ka_j$ with
$\alpha_{ij}\in K$ and  $a_j\in A$. Let
$E:=k\left(\left\{\alpha_{ij} \mid i=1,...,t\ ;\
j=1,...,n_i\right\}\right)$ and $Q:=P\cap(E\otimes_kA)$. Clearly,
$z_1,...,z_t\in Q$ and hence $P=Q(K\otimes_kA)=K\otimes_EQ.$ Apply
Lemma~\ref{s:1.1} to $K\otimes_kA\cong K\otimes_E(E\otimes_kA)$ to
obtain $\edim(K\otimes_kA)_P=\edim(E\otimes_kA)_Q.$ Now, let $F$ be
an intermediate field between $E$ and $K$ and $Q^{\prime}:=P\cap
(F\otimes_kA)$. Then
\begin{equation}\label{s:1.2eq1}P=Q^{\prime}(K\otimes_kA)=K\otimes_EQ^{\prime}\end{equation}
since $Q^{\prime}\cap(E\otimes_kA)=Q$. As above, Lemma~\ref{s:1.1}
leads to the conclusion.
\end{proof}

Next, we give the proof of the main theorem.
\bigskip

\noindent\emph{Proof of Theorem~\ref{s:1}}.
We proceed through three steps.

\noindent{\bf Step 1.} Assume that $K$ is an algebraic separable
extension field of $k$. We claim that
\begin{equation}\label{s:1eq1} P_P=\left(K\otimes_kp\right)_P.\end{equation}
Indeed,  set $S_p:= {\dfrac Ap}\setminus \{\overline 0\}$. The basic
fact $\dfrac P{K\otimes_kp}\cap\dfrac Ap=(\overline 0)$ yields
$$\frac{(K\otimes_kA)_P}{(K\otimes_kp)_P}\cong\left(K\otimes_k\frac Ap\right)_{\frac P{K\otimes_kp}}=
\Big(K\otimes_k\kappa_{A}(p)\Big)_{S_p^{-1}(\frac P{K\otimes_kp})}$$
where $K\otimes_k\kappa_{A}(p)$ is a zero-dimensional ring
\cite[Theorem 3.1]{S2}, reduced \cite[Chap. III, \S 15, Theorem
39]{ZS}, and hence von Neumann regular \cite[Ex. 22, p. 64]{K}. It
follows that $\Big(K\otimes_k\kappa_{A}(p)\Big)_{S_p^{-1}(\frac
P{K\otimes_kp})}$ is a field. Consequently, $(K\otimes_kp)_P=P_P$,
the unique maximal ideal of $(K\otimes_kA)_P$, proving our claim. By
(\ref{s:1eq1}) and Lemma~\ref{s:1.1}, we get
$\edim(K\otimes_kA)_P=\edim(A_p).$

\noindent{\bf Step 2.} Assume that $K$ is a finitely generated
separable extension field of $k$. Let $T=\{x_1,...,x_t\}$ be a
finite separating transcendence base of $K$ over $k$; that is, $K$
is algebraic separable over $k(T):=k(x_1,...,x_t)$. Let
$S:=k[T]\setminus\{0\}$ and notice that
$$K\otimes_kA\cong K\otimes_{k(T)}(k(T)\otimes_kA)\cong K\otimes_{k(T)}S^{-1}A[T].$$
Let $P\cap S^{-1}A[T]=S^{-1}P^{\prime}$ for some prime ideal
$P^{\prime}$ of $A[T]$. Note that $P^{\prime}\cap A=p$. Then, we
have
$$
\begin{array}{lll}
\edim(K\otimes_kA)_P    &=&\edim\left(K\otimes_{k(T)}S^{-1}A[T]\right)_P\\
                        &=&\edim\left(S^{-1}A[T]_{S^{-1}P^{\prime}}\right), \mbox { by Step 1}\\
                        &=&\edim\left(A[T]_{P^{\prime}}\right)\\
                        &=&\edim(A_p)+\htt\left(\dfrac{P^{\prime}}{p[T]}\right), \mbox { by Theorem~\ref{p:1}.}
\end{array}
$$
Moreover, note that
$$
\begin{array}{rcl}
K\otimes_k\dfrac Ap  &\cong  & K\otimes_{k(T)}\left(k(T)\otimes_k\dfrac Ap\right)\\
                    &\cong  & K\otimes_{k(T)}\dfrac {S^{-1}A[T]}{S^{-1}p[T]}
\end{array}
$$
 and
 $${\frac P{K\otimes_kp}\cap \frac {S^{-1}A[T]}{S^{-1}p[T]}=\frac {S^{-1}P^{\prime}}{S^{-1}p[T]}}$$
as $K\otimes_kp\cong
K\otimes_{k(T)}S^{-1}p[T]$ so that $(K\otimes_kp)\cap
S^{-1}A[T]=S^{-1}p[T]$. Therefore the integral extension
$\dfrac{S^{-1}A[T]}{S^{-1}p[T]}\hookrightarrow K\otimes_k{\dfrac
Ap}$ is flat and hence satisfies the Going-down property; that is, $
\htt\left(\dfrac{P^{\prime}}{p[T]}\right)    =
\htt\left(\dfrac{S^{-1}P^{\prime}}{S^{-1}p[T]}\right)=
\htt\left(\dfrac P{K\otimes_kp}\right). $ It follows that
$\edim(K\otimes_kA)_P    =\edim(A_p)\ +\ \htt{\left(\dfrac
P{K\otimes_kp}\right)}.$

\noindent{\bf Step 3.} Assume that $K$ is an arbitrary separable
extension field of $k$.  Then, by Lemma~\ref{s:1.2}, there exists a
finite extension field $E$ of $k$ contained in $K$ such that
$$\edim(K\otimes_kA)_{P}=\edim(E\otimes_kA)_{Q}$$
where $Q:=P\cap (E\otimes_kA)$. Let $\Omega$ denote the set of all
intermediate fields between $E$ and $K$. For each $F\in \Omega$,
note that $P=Q^{\prime}(K\otimes_kA)$, where
$Q^{\prime}:=P\cap(F\otimes_kA)$, as seen in (\ref{s:1.2eq1}); and
by Lemma~\ref{s:1.2} and Step 2, we obtain
\begin{equation}\label{s:1eq2}
\edim(K\otimes_kA)_P   = \edim(F\otimes_kA)_{Q^{\prime}}=
\edim(A_p)+\htt\left({\dfrac {Q^{\prime}}{F\otimes_kp}}\right).
\end{equation}
Further, as the ring extension ${F\otimes_k\dfrac Ap\hookrightarrow
K\otimes_k\dfrac Ap}$ satisfies the Going-down property, we get
\begin{equation}\label{s:1eq3}\htt\left(\frac{Q^{\prime}}{F\otimes_kp}\right)\leq\htt\left(\frac{P}{K\otimes_kp}\right).\end{equation}
Next let $K\otimes_kp\subseteq P_0\subsetneq
P_1\subsetneq...\subsetneq P_n=P$ be a chain of distinct prime
ideals of $K\otimes_kA$ such that $n:=\htt\left({\dfrac
P{K\otimes_kp}}\right)$. Let $t_i\in P_i\setminus P_{i-1}$ for each
$i=1,...,n$. One readily checks that there exists a finite extension
field $G$ of $k$ contained in $K$ such that, for each $i=1,...,n$,
$t_i\in G\otimes_kA$ and thus $t_i\in Q^{\prime}_i\setminus
Q^{\prime}_{i-1}$, where $Q^{\prime}_i:=P_i\cap (G\otimes_kA)$. Let
$H:=k(E,G)\in \Omega$ and $Q_i:=P_i\cap (H\otimes_kA)$ for each
$i=1,...,n$. Then $t_i\in Q_i\setminus Q_{i-1}$ for each
$i=1,...,n$. Therefore, we get the following chain of distinct prime
ideals in $H\otimes_kA$
$$H\otimes_kp\subseteq Q_0\subsetneq Q_1\subsetneq...\subsetneq Q_n=Q^{\prime}:=P\cap(H\otimes_kA).$$
It follows that $\htt\left(\frac{Q^{\prime}}{H\otimes_kp}\right)\geq
n$ and then (\ref{s:1eq3}) yields
$\htt\left(\frac{Q^{\prime}}{F\otimes_kp}\right)=\htt\left(\frac{P}{K\otimes_kp}\right).$
Further, $K\otimes_k\kappa_A(p)$ is regular since $K$ is separable
over K \cite[Lemma 6.7.4.1]{Gr}. Consequently, by (\ref{s:1eq2}), we
get
$$
\begin{array}{lll}
\edim(K\otimes_kA)_P    &=&\edim(A_p)\ +\ \htt{\left(\dfrac P{K\otimes_kp}\right)}\\
                        &=&\edim(A_p)\ +\  \edim\left(\Big(K\otimes_k\kappa_A(p)\Big)_{\frac{P_p}{K\otimes_kpA_p}}\right)
\end{array}
$$
completing the proof of the theorem.
\bigskip

As a direct application of Theorem~\ref{s:1}, we obtain the next
corollary on the (embedding) codimension which extends
Grothendieck's result on the transfer of regularity to tensor
products issued from finite extension fields \cite[Lemma
6.7.4.1]{Gr}. See also \cite{BK1}.

\begin{corollary}\label{s:2}
Let $K$ be a separable extension field of $k$ and $A$ a $k$-algebra
such that $K\otimes_kA$ is Noetherian. Let $P$ be a prime ideal of
$K\otimes_kA$ with $p:=P\cap A$. Then:
$$\cdim(K\otimes_kA)_{P}=\cdim(A_p).$$ In particular, $K\otimes_kA$ is regular if and only if $A$ is regular.
\end{corollary}

\begin{proof}
Combine Theorem~\ref{s:1} and (\ref{scte}).
\end{proof}

\section{Embedding dimension and codimension of tensor products of algebras with separable residue fields}\label{r}

\noindent This section deals with a more general setting (than in
Section \ref{s}); namely, we compute the embedding dimension of
localizations of the tensor product of two $k$-algebras $A$ and $B$
at prime ideals $P$ such that the residue field $\kappa_{B}(P\cap
B)$ is a separable extension of $k$.  The main result establishes an
analogue for an extended form of the ``special chain theorem" for
the Krull dimension which asserts that
\begin{equation}\label{r:eq1}
\begin{array}{rl}
\dim(A\otimes_kB)_{P}   &= \dim(A_{p})+\dim(B_{q})+\htt\left(\dfrac{P}{p\otimes_kB+A\otimes_kq}\right)\\
                        &= \dim(A_{p})+\dim(B_{q})+\dim\left(\big(\kappa_A(p)\otimes_k\kappa_B(q)\big)_{\frac{P(A_p\otimes_kB_q)}{pA_p\otimes_kB_{q}+
                        A_p\otimes_kqB_{q}}}\right).
\end{array}
\end{equation}
As an application, we formulate the (embedding) codimension of these
constructions with direct consequence on the transfer or defect of
regularity.

Here is the main result of this section.

\begin{thm}\label{r:1}
Let $A$ and $B$ be two $k$-algebras such that $A\otimes_kB$ is
Noetherian and let $P$ be a prime ideal of $A\otimes_kB$ with
$p:=P\cap A$ and $q:=P\cap B$. Assume $\kappa_B(q)$ is separable
over $k$. Then:
$$
\begin{array}{rl}
\edim(A\otimes_kB)_{P}\ =   &\edim(A_{p})\ +\ \edim(B_{q})\ +\ \htt{\left(\dfrac P{p\otimes_kB + A\otimes_kq}\right)}\\
                      \ =   &\edim(A_{p})\ +\ \edim(B_{q})\\
                            &\hspace{2.1cm}+\ \edim\left(\Big(\kappa_A(p)\otimes_k\kappa_B(q)\Big)_{\frac{P(A_p\otimes_kB_q)}{pA_p\otimes_kB_{q}+A_p\otimes_kqB_{q}}}\right)
\end{array}
$$
\end{thm}

\begin{proof}
Notice first that, as  $\kappa_B(q)$ is separable over $k$,
$\kappa_A(p)\otimes_k\kappa_B(q)$ is a regular ring and hence
$$
\begin{array}{rl}
\edim\left(\Big(\kappa_A(p)\otimes_k\kappa_B(q)\Big)_{\frac{P(A_p\otimes_kB_q)}{pA_p\otimes_kB_{q}+A_p\otimes_kqB_{q}}}\right)
                        &=\ \htt\left(\dfrac{P(A_p\otimes_kB_q)}{pA_p\otimes_kB_{q}+A_p\otimes_kqB_{q}}\right)\\
                        &=\ \htt\left(\dfrac P{p\otimes_kB + A\otimes_kq}\right).
\end{array}
$$
So, we only need to prove the first equality in the theorem and,
without loss of generality, we may assume that $(A,\n)$ and
$(B,{\m})$ are local $k$-algebras such that $A\otimes_kB$ is
Noetherian, $\dfrac{B}{\m}$ is a separable extension field of $k$,
and $P$ is a prime ideal of $A\otimes_kB$ with $P\cap A=\n$ and
$P\cap B={\m}$. Similar arguments used in the proof of
Lemma~\ref{s:1.2} show that there exists a finite extension field
$K$ of $k$ contained in $\dfrac{B}{\m}$ such that
$$\frac{P}{A\otimes_k{\m}}=Q\left(A\otimes_k\frac{B}{\m}\right)\cong Q\otimes_K\frac{B}{\m}$$
where $Q:=\dfrac{P}{A\otimes_k{\m}}\cap(A\otimes_kK)$. Since
${\dfrac B{{\m}}}$ is separable over $k$ and $K$ is a finitely
generated intermediate field, then $K$ is separably generated over
$k$ (cf. \cite[Chap. VI, Theorem 2.10 \& Definition 2.11]{H}). Let
$t$ denote the transcendence degree of $K$ over $k$ and let
$c_1,...,c_t\in B$ such that $\{\overline{c_1},...,\overline
{c_t}\}$ is a separating transcendence base of $K$ over $k$; i.e.,
$K$ is separable algebraic over $\Omega:=k\left(\overline{c_1},
...,\overline {c_t}\right)$. Also $c_1,...,c_t$ are algebraically
independent over $k$ with
\begin{equation}\label{r:1eq1}{\m}\cap k[c_1, ..., c_t]=(0).\end{equation}
So one may view $A\otimes_kk[c_1,...,c_t]\cong A[c_1,...,c_t]$ as a
polynomial ring in $t$ indeterminates over $A$. Set
$S:=k[c_1,...,c_t]\setminus \{0\}\ ;\ k(t):=k(c_1,...,c_t)\ ;\
A[t]:=A[c_1,...,c_t].$ Then, we have
\begin{equation}\label{r:1eq2}P\cap S=\m\cap S=\emptyset\ \text{ and }\ A\otimes_kS^{-1}B\cong S^{-1}A[t]\otimes_{k(t)}S^{-1}B.\end{equation}
Next, let $T:= {\frac P{A\otimes_k{\m}}}\cap(A\otimes_k\Omega)=Q\cap
(A\otimes_k\Omega)$ and consider the following canonical
isomorphisms of $k$-algebras $ \theta_1:A\otimes_k {\dfrac
{S^{-1}B}{S^{-1}{\m}}}\longrightarrow\big(A\otimes_kk(t)\big)\otimes_{k(t)}
{\dfrac{S^{-1}B}{S^{-1}{\m}}}$ and $\theta_2:A\otimes_k {\dfrac
B{{\m}}}\longrightarrow(A\otimes_k\Omega)\otimes_{\Omega} {\dfrac
B{{\m}}.}$ As $A\otimes_kK\cong(A\otimes_k\Omega)\otimes_{\Omega}K$,
by (\ref{s:1eq1}) we obtain
$Q_Q=(T\otimes_{\Omega}K)_Q=T(A\otimes_kK)_Q $ and hence
\begin{equation}\label{r:1eq3}
\begin{array}{lll}
\left(\dfrac P{A\otimes_k{\m}}\right)_{\frac P{A\otimes_k{\m}}} &=&
Q\left(A\otimes_k\dfrac B{{\m}}\right)_{\frac
P{A\otimes_k{\m}}}=Q_Q\left(A\otimes_k \dfrac
B{{\m}}\right)_{\left(\frac P{A\otimes_k{\m}}\right)_Q}\\
&=&T(A\otimes_{k}K)_Q\left(A\otimes_k \dfrac
B{{\m}}\right)_{\left(\frac P{A\otimes_k{\m}}\right)_Q}\\
&=&T(A\otimes_{k}K)\left(A\otimes_k \dfrac B{{\m}}\right)_{\frac
P{A\otimes_k{\m}}}=T\left(A\otimes_k \dfrac B{{\m}}\right)_{\frac
P{A\otimes_k{\m}}}\\
&=&\left(\theta_2^{-1}\left(\theta_2\left(T\left(A\otimes_k \dfrac
B{{\m}}\right)\right)\right)\right)_{\frac
P{A\otimes_k{\m}}}=\left(\theta_2^{-1} \left(T\otimes_{\Omega}\dfrac
B{{\m}}\right)\right)_{\frac P{A\otimes_k{\m}}}.
\end{array}
\end{equation}
Moreover, on account of (\ref{r:1eq1}) and by considering the
natural surjective homomorphism of $k$-algebras
$k[c_1,...,c_t]\stackrel {\varphi}\longrightarrow
k[\overline{c_1},...,\overline {c_t}]$ defined by
$\varphi(c_i)=\overline {c_i}$ for each $i=1,...,t$, we get
$k[c_1,...,c_t]\stackrel \varphi \cong k[\overline
{c_1},...,\overline {c_t}]$ inducing the following natural
isomorphism of extension fields
$\phi:=S^{-1}\varphi:k(t)\longrightarrow k(\overline
{c_1},...,\overline {c_t})=\Omega.$
Then, $\phi$ induces a structure of $k(t)$-algebras on $\Omega$ and
thus on $ {\dfrac B{{\m}}}$. We adopt a second structure of
$k(t)$-algebras on ${\dfrac B{{\m}}}$, inherited from the canonical
injection $k(t)\stackrel i\hookrightarrow S^{-1}B$. Indeed, consider
the following $k$-algebra homomorphisms $k(t)\stackrel {\overline
i}\longrightarrow{\frac
{S^{-1}B}{S^{-1}{\m}}}\stackrel{\gamma}\longrightarrow{\frac
B{{\m}}}$ defined by $\overline {i}(\alpha)=\overline {\alpha}$ for
each $\alpha\in k(t)$, and where $\gamma$ is the isomorphism of
$k$-algebras defined by $\gamma\Big({\overline{\frac
bs}\Big)=\frac{\overline {b}}{\overline s}}$ for each $b\in B$ and
each $s\in S$. It is easy to see that these two structures of
$k(t)$-algebras coincide on $ {\dfrac B{{\m}}}$. This is due to the
commutativity of the following diagram of homomorphisms of
$k$-algebras
$$\begin{array}{ccc}k(t)&\stackrel {\overline
i}\longrightarrow& {\dfrac {S^{-1}B}{S^{-1}{\m}}}\\
&&\\
&\phi\searrow&\downarrow\gamma\\
&& \\
&& {\dfrac B{{\m}}}
\end{array}$$
since, for each $\alpha:= {\frac {f}{s}}\in k(t)$ with $f\in
k[c_1,...,c_t]$ and $s\in S$, we have
$$(\gamma\circ\overline i)(\alpha)=\gamma(\overline {\alpha})=
\frac{\overline {f}}{\overline
{s}}=\dfrac{\varphi(f)}{\varphi(s)}=\phi(\alpha).$$ Now, consider
the following isomorphism of $k$-algebras
$$\psi:=\theta_2\circ(1_A\otimes_k\gamma)\circ\theta_1^{-1}:\big(A\otimes_kk(t)\big)\otimes_{k(t)} {\frac{S^{-1}B}{S^{-1}{\m}}\longrightarrow (A\otimes_k\Omega)\otimes_{\Omega}\frac B{{\m}}}$$
where, for each
$a\in A$, $\alpha\in k(t)$, $b\in B$, and $s\in S$, we have
$$\begin{array}{lll}
\psi\left((a\otimes_k\alpha)\otimes_{k(t)}\overline { {\frac
bs}}\right) &=&\theta_2\left((1_A\otimes_k\gamma)\Big
(a\otimes_{k}\overline{\alpha}\overline { {\frac bs}}\Big )\right)=\theta_2\left(a\otimes_k\gamma \Big (\overline{\alpha}\overline{ {\frac bs}}\Big )\right)\\
&=&\theta_2\left(a\otimes_k\Big((\gamma\circ \overline i)(\alpha)\gamma \big (\overline{ {\frac bs}}\big )\Big)\right)=
\theta_2\Big (a\otimes_k\Big (\phi(\alpha)\gamma \Big(\overline{ {\frac bs}}\Big )\Big )\Big )\\
&=&(a\otimes_k\overline 1)\otimes_{\Omega}\phi(\alpha)\gamma
\Big(\overline{ {\frac bs}}\Big )
=\left(a\otimes_k\phi(\alpha)\right)\otimes_{\Omega}\gamma \Big (\overline{ {\frac bs}}\Big )\\
&=&(1_A\otimes_k\phi)(a\otimes_k\alpha)\otimes_{\Omega}\gamma\Big
(\overline { {\frac bs}}\Big ).
\end{array}$$
Next, let $\delta:A\otimes_kS^{-1}B\longrightarrow
S^{-1}A[t]\otimes_{k(t)}S^{-1}B$ denote the canonical isomorphism of
$k$-algebras mentioned in (\ref{r:1eq2}) and let
$S^{-1}H:=S^{-1}P\cap S^{-1}A[t]$ where $H$ is a prime ideal of
$A[t]$ with $H\cap S=\emptyset$. Therefore
\begin{equation}\label{r:1eq4}
\begin{array}{lll}
\psi\left(S^{-1}H\otimes_{k(t)} {\dfrac {S^{-1}B}{S^{-1}{\m}}}\right)
&=&(1_A\otimes_k\phi)(S^{-1}H)\otimes_{\Omega}\gamma\Big ( {\dfrac {S^{-1}B}{S^{-1}{\m}}}\Big )\\
&=& (1_A\otimes_k\phi)(S^{-1}H)\otimes_{\Omega}  {\dfrac B{{\m}}}.
\end{array}
\end{equation}
{\bf Claim:}
$\delta(S^{-1}P)_{\delta(S^{-1}P)}=\Big(S^{-1}H\otimes_{k(t)}S^{-1}B+S^{-1}A[t]\otimes_{k(t)}S^{-1}{\m}\Big)_{\delta
(S^{-1}P)}$.

\noindent Indeed, consider the following commutative diagram (as
$\phi=\gamma\circ\overline i$)
$$\begin{array}{ccc}
S^{-1}(A\otimes_kB)=A\otimes_kS^{-1}B&\stackrel {1_A\otimes_k(\gamma\circ \pi)}\longrightarrow&A\otimes_k{\dfrac B{{\m}}}\\
&&\\
1_A\otimes_ki\uparrow&&\uparrow\\
&&\\
A\otimes_kk(t)&\stackrel
{1_A\otimes_k\phi}\longrightarrow&A\otimes_k\Omega
\end{array}$$
where $\pi:S^{-1}B\longrightarrow  {\frac {S^{-1}B}{S^{-1}{\m}}}$
denotes the canonical surjection (with $\pi\circ i=\overline i$) and
the vertical maps are the canonical injections. Also, it is worth
noting that $1_A\otimes_k\phi$ is an isomorphism of $k$-algebras.
Hence
\begin{equation}\label{r:1eq5}
\begin{array}{lll}
T&=& {\dfrac P{A\otimes_k{\m}}}\cap (A\otimes_k\Omega)=
(1_A\otimes_k\phi)\left(\left(\Big (1_A\otimes_k(\gamma\circ \pi)\Big)^{-1}\Big(\dfrac P{A\otimes_k{\m}}\Big)\right)\cap\big(A\otimes_kk(t)\big)\right)\\
&=&(1_A\otimes_k\phi)
\Big(S^{-1}P\cap\big(A\otimes_kk(t)\big)\Big)=(1_A\otimes_k\phi)
(S^{-1}P\cap S^{-1}A[t])\\
&=&(1_A\otimes_k\phi) (S^{-1}H).
\end{array}
\end{equation}
It follows, via  (\ref{r:1eq3}), (\ref{r:1eq5}), and (\ref{r:1eq4}), that
$$\begin{array}{lll}
\left(\dfrac P{A\otimes_k{\m}}\right)_{\frac P{A\otimes_k{\m}}}
&=& \theta_2^{-1}\left(T\otimes_{\Omega}\dfrac B{{\m}}\right)_{\frac P{A\otimes_k{\m}}}\\
&=&\theta_2^{-1}\left((1_A\otimes_k\phi)(S^{-1}H)\otimes_{\Omega} {\dfrac B{{\m}}}\right)_{\frac P{A\otimes_k{\m}}}\\
&=&\theta_2^{-1}\left(\psi\Big (S^{-1}H\otimes_{k(t)}{\dfrac {S^{-1}B}{S^{-1}{\m}}}\Big )\right)_{\frac P{A\otimes_k{\m}}}\\
&=&(1_A\otimes_k\gamma)\left(\theta_1^{-1}\Big(S^{-1}H\otimes_{k(t)} {\dfrac{S^{-1}B}{S^{-1}{\m}}}\Big )\right)_{\frac P{A\otimes_k{\m}}}.
\end{array}$$
Further, notice that $\dfrac
P{A\otimes_k{\m}}=(1_A\otimes_k\gamma)\left(\dfrac{S^{-1}P}{A\otimes_kS^{-1}{\m}}\right).$
Then the isomorphism $1_A\otimes_k\gamma$ yields the canonical
isomorphism of local $k$-algebras
$$(1_A\otimes_k\gamma)_P:\left(A\otimes_k {\frac{S^{-1}B}{S^{-1}{\m}}}\right)_{\frac{S^{-1}P}{A\otimes_kS^{-1}{\m}}}\longrightarrow
\left(A\otimes_k {\frac B{{\m}}}\right)_{\frac
P{A\otimes_k{\m}}}\mbox { with }$$

$$\begin{array}{lll}
(1_A\otimes_k\gamma)_P\left(\left(
{\dfrac{S^{-1}P}{A\otimes_kS^{-1}{\m}}}\right)_{\frac{S^{-1}P}{A\otimes_kS^{-1}{\m}}}\right)
&=& \left(\dfrac P{A\otimes_k{\m}}\right)_{\frac P{A\otimes_k{\m}}}\\
 &=&
(1_A\otimes_k\gamma)_P\left(\theta_1^{-1}\left(S^{-1}H\otimes_{k(t)}
{\dfrac{S^{-1}B}{S^{-1}{\m}}}\right)_{\frac{S^{-1}P}{A\otimes_kS^{-1}{\m}}}\right).
\end{array}$$
Therefore
\begin{equation}\label{r:1eq6}
\theta_1^{-1}\left(S^{-1}H\otimes_{k(t)}
{\frac{S^{-1}B}{S^{-1}{\m}}}\right)_{\frac{S^{-1}P}{A\otimes_kS^{-1}{\m}}}
=\left(
{\frac{S^{-1}P}{A\otimes_kS^{-1}{\m}}}\right)_{\frac{S^{-1}P}{A\otimes_kS^{-1}{\m}}}.
\end{equation}
Moreover, consider the following commutative diagram
$$\begin{array}{ccc}
A\otimes_kS^{-1}B&\stackrel {\delta}\longrightarrow&S^{-1}A[t]\otimes_{k(t)}S^{-1}B\\
\pi_1\downarrow&&\downarrow\pi_2\\
 A\otimes_k {\dfrac{S^{-1}B}{S^{-1}{\m}}}&\stackrel
{\theta_1}\longrightarrow&S^{-1}A[t]\otimes_{k(t)}
{\dfrac{S^{-1}B}{S^{-1}{\m}}}
\end{array}$$
where $\pi_1=1_A\otimes_k\pi$ and $\pi_2=1_{S^{-1}A[t]}\otimes_k\pi$
are the canonical surjective homomorphisms of $k$-algebras. Hence
$$\begin{array}{lll}
\pi_1^{-1}\left(\theta_1^{-1}\Big (S^{-1}H\otimes_{k(t)}
\dfrac{S^{-1}B}{S^{-1}{\m}}\Big )\right)
&=&(\theta_1\circ\pi_1)^{-1}\left(S^{-1}H\otimes_{k(t)} {\dfrac{S^{-1}B}{S^{-1}{\m}}}\right)\\
&=&(\pi_2\circ \delta)^{-1}\left(S^{-1}H\otimes_{k(t)}{\dfrac {S^{-1}B}{S^{-1}{\m}}}\right)\\
&=&\delta^{-1}\left(\pi_2^{-1}\Big(S^{-1}H\otimes_{k(t)} {\dfrac{S^{-1}B}{S^{-1}{\m}}}\Big )\right)\\
&=&\delta^{-1}\left(S^{-1}H\otimes_{k(t)}S^{-1}B+S^{-1}A[t]\otimes_{k(t)}S^{-1}{\m}\right)
\end{array}$$
so that
$$\begin{array}{lll}
\theta_1^{-1}\left(S^{-1}H\otimes_{k(t)}
{\dfrac{S^{-1}B}{S^{-1}{\m}}}\right)
&=&\pi_1\left(\delta^{-1}\Big(S^{-1}H\otimes_{k(t)}S^{-1}B+S^{-1}A[t]\otimes_{k(t)}S^{-1}{\m}\Big)\right)\\
&=& \dfrac
{\delta^{-1}\left(S^{-1}H\otimes_{k(t)}S^{-1}B+S^{-1}A[t]\otimes_{k(t)}S^{-1}{\m}\right)}{A\otimes_kS^{-1}{\m}}.
\end{array}$$
It follows, via (\ref{r:1eq6}), that
$$\begin{array}{lll}
\dfrac {S^{-1}P_{S^{-1}P}}{(A\otimes_kS^{-1}{\m})_{S^{-1}P}}
&=&\left(
{\dfrac{S^{-1}P}{A\otimes_kS^{-1}{\m}}}\right)_{\frac{S^{-1}P}{A\otimes_kS^{-1}{\m}}}
=\theta_1^{-1}\left(S^{-1}H\otimes_{k(t)} {\dfrac{S^{-1}B}{S^{-1}{\m}}}\right)_{\frac{S^{-1}P}{A\otimes_kS^{-1}{\m}}}\\
&=&\left( {\dfrac {\delta^{-1}\Big(S^{-1}H\otimes_{k(t)}S^{-1}B+S^{-1}A[t]\otimes_{k(t)}S^{-1}{\m}\Big)}{A\otimes_kS^{-1}{\m}}}\right)_{\frac{S^{-1}P}{A\otimes_kS^{-1}{\m}}}\\
 &=& {\dfrac
{\delta^{-1}\left(S^{-1}H\otimes_{k(t)}S^{-1}B+S^{-1}A[t]\otimes_{k(t)}S^{-1}{\m}\right)_{S^{-1}P}}{(A\otimes_kS^{-1}{\m})_{S^{-1}P}}}
\end{array}$$
and thus
$S^{-1}P_{S^{-1}P}=\delta^{-1}\Big(S^{-1}H\otimes_{k(t)}S^{-1}B+S^{-1}A[t]\otimes_{k(t)}S^{-1}{\m}\Big)_{S^{-1}P}.$
Also, note that the isomorphism of $k$-algebras $\delta$ induces the
isomorphism of local $k$-algebras\\
$\delta_P:(A\otimes_kS^{-1}B)_{S^{-1}P}\longrightarrow(S^{-1}A[t]\otimes_{k(t)}S^{-1}B)_{\delta(S^{-1}P)}.$
Hence
$$\begin{array}{lll}
\delta_P^{-1}\left(\delta(S^{-1}P)_{\delta(S^{-1}P)}\right)
&=&S^{-1}P_{S^{-1}P}\\
&=&\delta_P^{-1}\left(\Big(S^{-1}H\otimes_{k(t)}S^{-1}B+S^{-1}A[t]\otimes_{k(t)}S^{-1}{\m}\Big)_{\delta
(S^{-1}P)}\right)
\end{array}$$
so that
$\delta(S^{-1}P)_{\delta(S^{-1}P)}=\Big(S^{-1}H\otimes_{k(t)}S^{-1}B+S^{-1}A[t]\otimes_{k(t)}S^{-1}{\m}\Big)_{\delta
(S^{-1}P)}$ proving the claim.

It follows, by Lemma~\ref{s:1.1} applied to
$S^{-1}A[t]\otimes_{k(t)}S^{-1}B$, that
$$\mu_{\delta (S^{-1}P)}(S^{-1}{\m}S^{-1}B_{S^{-1}{\m}})=\edim(S^{-1}B_{S^{-1}{\m}})=\edim(B)$$
so that, by Proposition~\ref{f:1}, we have

$$\begin{array}{lll}
\edim(A\otimes_kB)_P
&=&\edim\left((A\otimes_kS^{-1}B)_{S^{-1}P}\right)
=\edim\left((S^{-1}A[t]\otimes_{k(t)}S^{-1}B)_{\delta(S^{-1}P)}\right)\\
&=&\mu_{\delta
(S^{-1}P)}(S^{-1}{\m}S^{-1}B_{S^{-1}{\m}})+\edim\left(\Big
(S^{-1}A[t]\otimes_{k(t)}
{\dfrac{S^{-1}B}{S^{-1}{\m}}}\Big)_{\frac{\delta(S^{-1}P)}{S^{-1}A[t]\otimes_{k(t)}S^{-1}{\m}}}\right)\\
&=&\edim(B)+\edim\left(\Big(S^{-1}A[t]\otimes_{k(t)}
{\dfrac{S^{-1}B}{S^{-1}{\m}}}\Big
)_{\frac{\delta(S^{-1}P)}{S^{-1}A[t]\otimes_{k(t)}S^{-1}{\m}}}\right)\\
&=&\edim(B)+\edim\left(\Big(A\otimes_k {\dfrac
{S^{-1}B}{S^{-1}{\m}}}\Big)_{\frac
{S^{-1}P}{A\otimes_kS^{-1}{\m}}}\right).
\end{array}$$

Finally, as ${\dfrac {S^{-1}B}{S^{-1}{\m}}\cong \dfrac B{{\m}}}$ is
a separable extension field of $k$, we get, by Theorem~\ref{s:1},
that
$$\begin{array}{lll}
\edim(A\otimes_kB)_P &=&\edim(A)+\edim(B)+\htt\left(
{\dfrac{S^{-1}P/(A\otimes_kS^{-1}{\m})}{\n\otimes_k(S^{-1}B/S^{-1}{\m})}}\right)\\
&=&\edim(A)+\edim(B)+\htt\left(
{\dfrac{S^{-1}P}{\n\otimes_kS^{-1}B+A\otimes_kS^{-1}{\m}}}\right)\\
&=&\edim(A)+\edim(B)+\htt\left( {\dfrac
P{\n\otimes_kB+A\otimes_k{\m}}}\right)
\end{array}$$
completing the proof of the theorem.
\end{proof}

As a direct application of Theorem~\ref{r:1}, we obtain the next
corollary on the (embedding) codimension which recovers known
results on the transfer of regularity to tensor products over
perfect fields \cite[Theorem 6(c)]{TY} and, more generally, to
tensor products issued from residually separable extension fields
\cite[Theorem 2.11]{BK1}. Recall that a $k$-algebra $R$ is said to
be residually separable, if $\kappa_{R}(p)$ is separable over $k$
for each prime ideal $p$ of $R$.

\begin{corollary}\label{r:2}
Let $A$ and $B$ be two $k$-algebras such that $A\otimes_kB$ is
Noetherian and let $P$ be a prime ideal of $A\otimes_kB$ with
$p:=P\cap A$ and $q:=P\cap B$. Assume $\kappa_B(q)$ is separable
over $k$. Then:
$$\cdim(A\otimes_kB)_{P}=\cdim(A_p)+\cdim(B_q).$$
\end{corollary}

\begin{proof}
Combine Theorem~\ref{r:1} and (\ref{r:eq1}).
\end{proof}

Note that if $k$ is perfect, then every $k$-algebra is residually
separable. Now, if $k$ is an arbitrary field, one can easily provide
original examples of residually separable $k$-algebras through
localizations of polynomial rings or pullbacks \cite{BG,BR}. For
instance, let $X$ be an indeterminate over $k$ and $K\subseteq L$
two separable extensions of $k$. Then, the one-dimensional local
$k$-algebras $R:=K+XL[X]_{(X)}\subseteq S:=L[X]_{(X)}$ are
residually separable since the extensions
$k\subseteq\kappa_{R}\left(XL[X]_{(X)}\right)=K\subseteq\kappa_{S}\left(XL[X]_{(X)}\right)=L\subset\kappa_R(0)=\kappa_S(0)=L(X)$
are separable over $k$ by Mac Lane's Criterion and transitivity of
separability. Also, similar arguments show that the two-dimensional
local $k$-algebra $R':=R+YL(X)[Y]_{(Y)}$  is residually separable,
where $Y$ is an indeterminate over $k$. Therefore, one may reiterate
the same process to build residually separable $k$-algebras of
arbitrary Krull dimension.

\begin{corollary}\label{r:3}
Let $A$ be a finitely generated $k$-algebra and $B$ a residually
separable $k$-algebra. Let $P$ be a prime ideal of $A\otimes_kB$
with $p:=P\cap A$ and $q:=P\cap B$. Then:
$$\cdim(A\otimes_kB)_{P}=\cdim(A_p)+\cdim(B_q).$$ In particular, $A\otimes_kB$ is regular if and only if so are $A$ and $B$.
\end{corollary}

\begin{corollary}\label{4:3}
Let $k$ be an algebraically closed field, $A$ a finitely generated
$k$-algebra, $p$ a maximal ideal of $A$, and $B$ an arbitrary
$k$-algebra. Let $P$ be a prime ideal of $A\otimes_kB$ such that
$P\cap A=p$ and set $q:=P\cap B$. Then:
$$\cdim(A\otimes_kB)_{P}=\cdim(A_p)+\cdim(B_q).$$
\end{corollary}



\begin{thebibliography}{99}

\bibitem{ABDFK} D. F. Anderson, A. Bouvier, D. E. Dobbs, M. Fontana, and S. Kabbaj, On Jaffard domains, Expo. Math. \textbf{6} (2) (1988) 145--175.

\bibitem{BG}    E. Bastida and R. Gilmer, Overrings and divisorial ideals of rings of the form $D+M$, Michigan Math. J. \textbf{20} (1973) 79--95.

\bibitem{BDK} S. Bouchiba, D. E. Dobbs, and S. Kabbaj, On the prime ideal structure of tensor products of algebras, J. Pure Appl. Algebra \textbf{176} (2002) 89--112.

\bibitem{BCM} S. Bouchiba, J. Conde-Lago, and J. Majadas, Cohen-Macaulay, Gorenstein, complete intersection and regular defect for the tensor product of algebras, Preprint  $\langle$arXiv:1512.02804$\rangle$

\bibitem{BGK1}  S. Bouchiba, F. Girolami and S. Kabbaj, The dimension of tensor products of AF-rings, pp. 141--154, Lecture Notes in Pure Appl. Math., Vol. 185, Dekker, New York, 1997.

\bibitem{BGK2}  S. Bouchiba, F. Girolami and S. Kabbaj, The dimension of tensor products of $k$-algebras arising from pullbacks, J. Pure Appl. Algebra \textbf{137} (1999) 125--138.

\bibitem{BK2} S. Bouchiba and S. Kabbaj, Tensor products of Cohen-Macaulay rings. Solution to a problem of Grothendieck, J. Algebra \textbf{252} (2002) 65--73.

\bibitem{BK1} S. Bouchiba and S. Kabbaj, Regularity of tensor products of $k$-algebras, Math. Scand. 115 (1) (2014) 5--19.

\bibitem{B4-7} N. Bourbaki, Alg\`ebre, Chapitres 4--7, Masson, Paris, 1981.

\bibitem{BAC8-9} N. Bourbaki, Alg\`ebre Commutative, Chapitres 8--9, Masson, Paris, 1981.

\bibitem{BDF} A. Bouvier, D. E. Dobbs, and M. Fontana, Universally catenarian integral domains, Adv. in Math. \textbf{72} (2) (1988) 211--238.

\bibitem{BHMR} J. Brewer, P. Montgomery, E. Rutter and W. Heinzer, Krull dimension of polynomial rings,  pp. 26--45, Lecture Notes in Math., Vol. 311, Springer, Berlin, 1973.

\bibitem{BR}   J. W. Brewer and E. A. Rutter, $D+M$ constructions with general overrings, Michigan Math. J. \textbf{23} (1976) 33--42.

\bibitem{BH}  W. Bruns and J. Herzog, Cohen-Macaulay rings, Cambridge University Press, Cambridge, 1993.

\bibitem{E} D. Eisenbud, Commutative algebra with a view toward algebraic geometry, GTM, vol. 150, Springer-Verlag, New York, 1995.

\bibitem{Fer}   D. Ferrand, Monomorphismes et morphismes absolument plats, Bull. Soc. Math. France \textbf{100} (1972) 97--128.

\bibitem{FK}  M. Fontana and S. Kabbaj, Essential domains and two conjectures in dimension theory, Proc. Amer. Math. Soc. \textbf{132} (9) (2004) 2529--2535.

\bibitem{Gr}    A. Grothendieck, El\'ements de g\'eom\'etrie alg\'ebrique, Institut des Hautes Etudes Sci. Publ. Math. No. 24, Bures-sur-yvette, 1965.

\bibitem{HTY}    H. Haghighi, M. Tousi, and S. Yassemi, Tensor product of algebras over a field, in: ``Commutative algebra.
Noetherian and non-Noetherian perspectives," pp. 181--202, Springer,
New York, 2011.

\bibitem{HJ}    C. Huneke and D. A. Jorgensen, Symmetry in the vanishing of Ext over Gorenstein rings, Math. Scand. \textbf{93} (2) (2003) 161--184.

\bibitem{H} T. Hungerford, Algebra, Springer-Verlag, New York, 1974.

\bibitem{J} P. Jaffard,  Th\'eorie de la dimension dans les anneaux de polyn\^omes, M\'em. Sc. Math., 146, Gauthier-Villars, Paris
(1960).

\bibitem{Jor}    D. A. Jorgensen, On tensor products of rings and extension conjectures, J. Commut. Algebra \textbf{1} (4) (2009) 635--646.

\bibitem{K}    I. Kaplansky, Commutative rings, University of Chicago Press, Chicago, 1974.

\bibitem{Maj}    J. Majadas, On tensor products of complete intersections, Bull. Lond. Math. Soc.  45 (6)  (2013) 1281--1284.

\bibitem{M} H. Matsumura, Commutative ring theory, Cambridge University Press, Cambridge, 1986.

\bibitem{Na} M. Nagata, Local rings, Robert E. Krieger Publishing Co., Huntington, N.Y., 1975.

\bibitem{Rot} J. J. Rotman, An introduction to homological algebra, Second edition, Universitext, Springer, New York, 2009.

\bibitem{S}    R.Y. Sharp, Simplifications in the theory of tensor products of field extensions, J. London Math. Soc.  \textbf{15}
(1977) 48--50.

\bibitem{S2}    R.Y. Sharp, The dimension of the tensor product of two field extensions, Bull. London Math. Soc. \textbf{9} (1977) 42--48.

\bibitem{S3}    R.Y. Sharp, The effect on associated prime ideals produced by an extension of the base field, Math. Scand. \textbf{38} (1976) 43--52.

\bibitem{Sw}    M. E. Sweedler, When is the tensor product of algebras local? Proc. Amer. Math. Soc.  \textbf{48}  (1975) 8--10.

\bibitem{TY}    M. Tousi and S. Yassemi, Tensor products of some special rings, J. Algebra \textbf{268} (2003) 672--676.

\bibitem{V}     P. Vamos, On the minimal prime ideals of a tensor product of two fields, Math. Proc. Camb. Phil. Soc. \textbf{84} (1978) 25--35.

\bibitem{W}  A.R. Wadsworth, The Krull dimension of tensor products of commutative algebras over a field, J. London Math. Soc. \textbf{19} (1979) 391--401.

\bibitem{WITO}  K. Watanabe, T. Ishikawa, S. Tachibana, and K. Otsuka, On tensor products of Gorenstein rings, J. Math. Kyoto Univ. \textbf{9} (1969) 413--423.\par

\bibitem{ZS}  O. Zariski and P. Samuel, Commutative algebra, Vol. I, Van Nostrand, Princeton, 1960.\par
\end{thebibliography}
\end{document}